\documentclass[11pt,draft]{amsart}

\newtheorem{theorem}{Theorem}[section]
\newtheorem{lemma}[theorem]{Lemma}

\newtheorem{Proposition}[theorem]{Proposition}
\theoremstyle{definition}

\newtheorem{remark}[theorem]{Remark}

\RequirePackage{amsmath}
\RequirePackage{amssymb}
\RequirePackage{amsthm}
 \RequirePackage{epsfig}

\begin{document}

\date{}

\title[Improved sum-product estimates]
{Slightly improved sum-product estimates in fields of prime order}

\author{Liangpan Li}

\address{Department of Mathematics, Shanghai Jiao Tong University,
Shanghai 200240,  China $\&$ Department of Mathematics, Texas State
University, San Marcos, Texas 78666, USA }
 \email{liliangpan@yahoo.com.cn}

\subjclass[2000]{11B75}

\keywords{sum-product estimates}

\date{}

\begin{abstract}
Let $\mathbb{F}_p$ be the field of residue classes modulo a prime
number $p$ and let $A$ be a nonempty subset of $\mathbb{F}_p$. In
this paper we show that if $|A|\preceq p^{0.5}$, then
\[
\max\{|A\pm A|,|AA|\}\succeq|A|^{13/12};\]
if $|A|\succeq p^{0.5}$,
then
\[ \max\{|A\pm A|,|AA|\}\succapprox
\min\{|A|^{13/12}(\frac{|A|}{p^{0.5}})^{1/12},|A|(\frac{p}{|A|})^{1/11}\}.\]
These results slightly improve the estimates of Bourgain-Garaev and
Shen. Sum-product estimates on different sets are also considered.

\end{abstract}

\maketitle

\section{Introduction}

Let $\mathbb{F}_p$ be the field of residue classes modulo a prime
number $p$ and let $A, B$ be two nonempty subsets of $\mathbb{F}_p$.
Define the sum set, difference set and  product set of $A$ and $B$
respectively by
\begin{align*}
A+B&=\{a+b: a\in A, b\in B\},\\
A-B&=\{a-b: a\in A, b\in B\},\\
AB&=\{ab: a\in A, b\in B\}.
\end{align*}
From the work of Bourgain, Katz, Tao \cite{BourgainKatzTao} and
Bourgain, Glibichuk, Konyagin \cite{BourgainGlibichukKonyagin}, it
is known that if $|A|\preceq p^{\delta}$ (see Section 2 for the
definitions of $\preceq,\succeq,\precapprox,\succapprox$ and
$\sim$), where $\delta<1$, then one has the sum-product estimate
\[\max\{|A+A|,|AA|\}\succeq|A|^{1+\epsilon}\]
for some $\epsilon=\epsilon(\delta)>0$. These kinds of results have
found many important applications in various areas of mathematics,
and people want to know some quantitative relationships between
$\delta$ and $\epsilon$ in certain ranges of $|A|$. For the case
$|A|\succeq p^{0.5}$, the pioneer work was due to Hart, Iosevich and
Solymosi \cite{HartIosevichSolymosi} via Kloosterman sums. See
\cite{chang0,Garaevlargesubset,Vu} for further improvements. Note
also all lower bounds in
\cite{chang0,Garaevlargesubset,HartIosevichSolymosi,Vu} are trivial
if $|A|\sim p^{0.5}$.

 In the very beginning of the
story for the case  $|A|\preceq p^{0.5}$, Garaev
\cite{Garaevoriginal} proved
\[\max\{|A+A|,|AA|\}\succapprox|A|^{15/14},\]
which was immediately improved by Katz and Shen \cite{KatzShen} with
a refinement of the Pl\"{u}nnecke-Ruzsa inequality to
\[\max\{|A+A|,|AA|\}\succapprox|A|^{14/13}.\]
Later on, Bourgain and Garaev \cite{BourgainGaraev} considered
difference-product estimates and proved
\begin{equation}\label{bourgain}
\max\{|A-A|,|AA|\}\succeq\frac{|A|^{13/12}}{(\log_2|A|)^{4/11}},
\end{equation}
which was slightly improved  by Shen \cite{Shen,Shen2} with elegant
covering arguments to
\begin{equation}\label{shen}
\max\{|A\pm A|,|AA|\}\succeq\frac{|A|^{13/12}}{(\log_2|A|)^{1/3}}.
\end{equation}
With a technique of Chang \cite{Changmeichu}, we can completely drop
the logarithmetic term
 in (\ref{shen}).

\begin{theorem}\label{Theorem1312}
 Suppose $A\subset \mathbb{F}_p$ with $|A|\preceq p^{0.5}$. Then
\[ \max\{|A\pm A|,|AA|\}\succeq|A|^{13/12}.\]
\end{theorem}

Note  the Bourgain-Garaev estimate (\ref{bourgain}) also holds for
$p^{0.5}\preceq|A|\preceq p^{12/23}$ (see \cite{BourgainGaraev}). In
these ranges and beyond, our next result says that:

\begin{theorem}\label{Theoremlarge}
 Suppose $A\subset \mathbb{F}_p$ with $|A|\succeq p^{0.5}$. Then
\[ \max\{|A\pm A|,|AA|\}\succapprox\min\{|A|^{13/12}(\frac{|A|}{p^{0.5}})^{1/12},|A|(\frac{p}{|A|})^{1/11}\}.\]
\end{theorem}

To compare, if $|A|\preceq p^{12/23}$, then
$|A|^{13/12}\preceq|A|(\frac{p}{|A|})^{1/11}$. Particularly, if
$|A|\sim p^{35/68}$, then $ \max\{|A\pm
A|,|AA|\}\succapprox|A|^{38/35}$. This shows Theorem
\ref{Theoremlarge} is an improvement of (\ref{bourgain}).

Similarly, we may consider sum-product estimates on different sets
in $\mathbb{F}_p$. Bourgain \cite{Bourgaindifferentsets} proved that
if $p^{1-\delta}\preceq|B|\preceq|A|\preceq p^{\delta}$, then for
some $\epsilon=\epsilon(\delta)>0$,
\[ \max\{\frac{|A+B|}{|A|},\frac{|AB|}{|A|}\}\succeq p^{\epsilon}.\]
Shen \cite{Shendifferentsets} quantitatively proved that (a) if
$|B|\preceq|A|\preceq p^{0.5}$, then
\[ \max\{\frac{|A+B|}{|A|},\frac{|AB|}{|A|}\}\succapprox (\frac{|B|^{14}}{|A|^{13}})^{1/18};\]
(b) if $|B|\sim|A|\preceq p^{0.5}$, then
\[ \max\{|A+B|,|AB|\}\succapprox|A|^{15/14}.\]
 We can also give some improvements.
\begin{theorem}\label{main11}
 Suppose $A,B\subset \mathbb{F}_p$ with $|B|\preceq|A|\preceq p^{0.5}$. Then
\[ \max\{\frac{|A+B|}{|A|},\frac{|AB|}{|A|}\}\succapprox(\frac{|B|^{6}}{|A|^{5}})^{1/14}.\]
\end{theorem}

\begin{theorem}\label{main33}
 Suppose $A,B\subset \mathbb{F}_p$ with $|A|\succeq|B|\succeq p^{0.5}$. Then
\[ \max\{\frac{|A+B|}{|A|},\frac{|AB|}{|A|}\}\succapprox\min\big\{
\big(\frac{|B|^7}{|A|^5p^{0.5}}\big)^{1/14},\big(\frac{|B|^3p}{|A|^4}\big)^{1/12}\big\}.\]
\end{theorem}

\begin{theorem}\label{main22}
 Suppose $A,B\subset \mathbb{F}_p$ with $|B|\sim|A|\preceq p^{0.5}$. Then
\[ \max\{|A+B|,|AB|\}\succeq|A|^{15/14}.\]
\end{theorem}

\section{Notations and Lemmas}

Throughout this paper $A$ will denote a fixed nonempty set in
$\mathbb{F}_p$. For $B$, any set, we will denote by $|B|$ its
cardinality . Whenever $E$ and $F$ are quantities we will use
$E\preceq F$ or $F\succeq E$ to mean $E\leq CF$, where the constant
$C$ is universal (i.e. independent of $A$ and $p$). We will use
$E\precapprox F$ or $F\succapprox E$ to mean $E\leq
C(\log|A|)^{\alpha}F$, where the universal constants $C$ and
$\alpha$ may vary from line to line. Besides, $E\sim F$ means
$E\preceq F$ and $F\preceq E$.


For  $Y,Z\subset \mathbb{F}_p$, denote by $E^{+}(Y,Z)$
 the additive energy between  $Y$ and $Z$, that is,
\[E^{+}(Y,Z)=\sum_{x\in Y}\sum_{y\in Y}|(x+Z)\cap (y+Z)|;\]
denote by $E^{\times}(Y,Z)$
 the multiplicative energy between  $Y$ and $Z$, that is,
\[E^{\times}(Y,Z)=\sum_{x\in Y}\sum_{y\in Y}|xZ\cap yZ|.\]
It is well-known \cite{tao} that
\[E^{\odot}(Y,Z)\geq\frac{|Y|^2|Z|^2}{|Y\odot Z|},\]
where $\odot\in\{+,\times\}$.

In the following we will give some preliminary lemmas.
 Lemma \ref{coverlemma} may be found
in \cite{Shendifferentsets,Shen,Shen2}, while  Lemma
\ref{ImprovedPlunneckeRuzsa} in \cite{Gyarmati,KatzShen}. Lemma
\ref{Konyaginababcd}, following from the work of Glibichuk and
Konyagin \cite{Glibichuk,KonyaginGlibichuk} on additive properties
of product sets, was proved in
\cite{BourgainGaraev,Garaevoriginal,KatzShen,Shen}. Since the author
have not found a proof of  Lemma \ref{lemmallp} in some popular
references, we include a short proof here. Lemma \ref{lemmachang} is
due to Chang \cite{Changmeichu}, whereas we present a slightly
different variant.

\begin{lemma}\label{coverlemma}
  Suppose $B_1,B_2\subset \mathbb{F}_p$. Then there exist $\preceq\min\{\frac{|B_1+ B_2|}{|B_2|},\frac{|B_1-B_2|}{|B_2|}\}$ translates of $B_2$
  such
that these copies can cover (in cardinality) $99\%$ of $B_1$.
\end{lemma}

\begin{lemma}\label{ImprovedPlunneckeRuzsa}
 Suppose $B_0,B_1,\ldots,B_k\subset\mathbb{F}_p$. Given any $\epsilon\in(0,1)$,
there exist a universal constant $C_{k,\epsilon}$  and
 an $X\subset B_0$ with $|X|\geq(1-\epsilon)|B_0|$ such that
\[
|X+B_1+B_2+\cdots+B_k|\leq
C_{k,\epsilon}\cdot\big(\prod_{i=1}^k\frac{|B_i+B_0|}{|B_0|}\big)\cdot|X|.
\]
\end{lemma}

\begin{lemma}\label{Konyaginababcd}
Suppose $A_1\subset \mathbb{F}_p$ with
$\frac{A_1-A_1}{A_1-A_1}\neq\mathbb{F}_p$. Then (1) $|A_1|\preceq
p^{0.5}$; (2) there exist fixed elements $a_1,b_1,c_1,d_1\in A_1$
($a_1\neq b_1$)
 such that for any $A'\subset A_1$ with
$|A'|\succeq|A_1|$,
\[
|(b_1-a_1)A'+(b_1-a_1)A'+(d_1-c_1)A'|\succeq|A_1|^2;\] (3) there
exist
 fixed elements $a_2,b_2,c_2,d_2\in A_1$ ($a_2\neq b_2$)
 such that for any $A''\subset A_1$ with
$|A''|\succeq|A_1|$,
\[
|(b_2-a_2)A''-(b_2-a_2)A''+(d_2-c_2)A''|\succeq|A_1|^2.
\]
\end{lemma}

\begin{lemma}\label{lemmallp} Suppose $A_1\subset \mathbb{F}_p$ with
  $\frac{A_1-A_1}{A_1-A_1}=\mathbb{F}_p$.
Then there exist fixed elements $a_3,b_3,c_3,d_3\in A_1$ ($a_3\neq
b_3$) such that for any $A'''\subset A_1$ with $|A'''|\succeq|A_1|$,
\[|(b_3-a_3)A'''+(d_3-c_3)A'''|\succeq\min\{|A_1|^2,p\}.\]
\end{lemma}

\begin{proof}
 There exists   $\xi\in\mathbb{F}_p^{\ast}=\mathbb{F}_p\backslash\{0\}$  (cf. Formula (5) in \cite{Glibichuk}) such that
\[E^{+}(A_1,\xi A_1)\leq|A_1|^2+\frac{|A_1|^4}{p-1}.\]
 Since $\frac{A_1-A_1}{A_1-A_1}=\mathbb{F}_p$, we can write
$\xi=\frac{d_3-c_3}{b_3-a_3}$ for some $a_3,b_3,c_3,d_3\in A_1$.
Thus
\[|A'''+\xi A'''|\geq\frac{|A'''|^4}{E^{+}(A''',\xi A''')}\geq\frac{|A'''|^4}{E^{+}
(A_1,\xi A_1)}\succeq\frac{|A_1|^4}{E^{+}(A_1,\xi
A_1)}\succeq\min\{|A_1|^2,p\}.\] This proves the lemma.
\end{proof}

\begin{lemma}\label{lemmachang}
Suppose $Y,Z\subset \mathbb{F}_p$. Choose a fixed element $y_0\in Y$
so that \[\sum_{y\in Y}|y_0Z\cap
yZ|\geq\frac{E^{\times}(Y,Z)}{|Y|}.\]  For each
$j\leq\lceil\log_2|Z|\rceil$, let $Y_j$ be the set of all $y\in Y$
for which $|y_0Z\cap yZ|\in N_j$, where $N_1=\{1,2\}, N_2=\{3,4\},
N_3=\{5,6,7,8\}, N_4=\{9,10,11,12,13,14,15,16\}$, $\ldots$. Then
\[\max_{j}16^j|Y_j|^3\succeq\frac{E^{\times}(Y,Z)^4}{|Y|^4|Z|}.\]

\end{lemma}

\begin{proof}
Define $j_s=\max\{j:|Y_j|\in N_s\}$ for each $s\leq
\lceil\log_2|Z|\rceil$ (assume $\max\emptyset=0$). Clearly,
\[\sum_{s:j_s\geq1}2^{j_s}2^s\sim
\sum_{j=1}^{\lceil\log_2|Z|\rceil}2^j |Y_j| \sim\sum_{y\in
Y}|y_0Z\cap yZ|.
\]
Note also \[\label{30}\sum_{s:j_s\geq1}2^{j_s}2^s\leq \big(\max_{s:
j_s\geq1}2^{j_s}2^{0.75s}\big)
\sum_{s=1}^{\lceil\log_2|Z|\rceil}2^{0.25s}
\preceq\big(\max_j2^j|Y_j|^{0.75}\big)\cdot|Z|^{0.25}.\] This proves
the lemma.
\end{proof}

\section{Sum-product estimates  on small sets}

In this section we prove Theorem \ref{Theorem1312}. Suppose
$A\subset \mathbb{F}_p$ with $|A|\preceq p^{0.5}$. Applying Lemma
\ref{ImprovedPlunneckeRuzsa} twice with
$\epsilon=\frac{\sqrt{2}}{2}$, one can find a subset $Z\subset A$
with $|Z|\geq\frac{|A|}{2}$ such that
\begin{equation}\label{sss}|Z\pm Z\pm Z\pm Z|\leq|Z\pm A\pm A\pm A|\preceq
(\frac{|A\pm A|}{|A|})^3|Z|\preceq\frac{|A\pm
A|^3}{|A|^2}.\end{equation} Choose a fixed element $z_0\in Z$ so
that \[\sum_{z\in Z}|z_0Z\cap zZ|\geq\frac{E^{\times}(Z,Z)}{|Z|}.\]
For each $j\leq\lceil\log_2|Z|\rceil$, let $Z_j$ be the set of all
$z\in Z$ for which $|z_0Z\cap zZ|\in N_j$  (see Lemma
\ref{lemmachang} for the meaning of $N_j$). Then we can deduce from
\cite{Shen,Shen2} or mimic the proof of Proposition
\ref{3333proposition} (see also Formula (\ref{important}) in Section
5) to know that
\begin{equation}\label{2222proposition}
\max_j16^{j}|Z_j|^3\preceq|Z\pm Z|^5\cdot|Z\pm Z\pm Z\pm Z|.
\end{equation}
By Lemma  \ref{lemmachang},
\begin{equation}\label{llpllp}
\max_j16^{j}|Z_j|^3\succeq\frac{E^{\times}(Z,Z)^4}{|Z|^5}\geq\frac{|Z|^{11}}{|ZZ|^4}\succeq\frac{|A|^{11}}{|AA|^4}.
\end{equation} Combining (\ref{sss}), (\ref{2222proposition}) and
(\ref{llpllp}) yields
\[|A\pm A|^8|AA|^4\succeq|A|^{13}.\]
This proves Theorem \ref{Theorem1312}.

\begin{remark} To establish (\ref{bourgain}),
Bourgain and Garaev \cite{BourgainGaraev} actually proved that for
any $A\subset \mathbb{F}_p$ one has
\[E^{\times}(A,A)^4\preceq\big(|A-A|+\frac{|A|^3}{p}\big)\cdot|A|^5\cdot|A-A|^4\cdot|A+A-A-A|\cdot(\log_2|A|)^4.\]
Particularly, if $|A|\preceq p^{0.5}$, then
\begin{equation}\label{iioo}E^{\times}(A,A)^4\preceq|A|^5\cdot|A-A|^5\cdot|A+A-A-A|\cdot (\log_2|A|)^4.\end{equation}
Based on the arguments in \cite{BourgainGaraev} and this section one
can drop the logarithmetic term in (\ref{iioo}):
\begin{equation}\label{jjjjjjjj}
E^{\times}(A,A)^4\preceq|A|^5\cdot|A-A|^5\cdot|A+A-A-A|.
\end{equation}
Besides, two byproducts of the proof of Theorem \ref{Theorem1312}
are the estimates (suppose $|A|\preceq p^{0.5}$):
\begin{align}
E^{\times}(A,A)^4&\preceq|A|^5\cdot|A\pm A|^5\cdot|A\pm A\pm A\pm
A|,\label{jjjjj}\\
E^{\times}(A,A)^4&\preceq|A|^3\cdot|A\pm A|^8.
\end{align}


\end{remark}

\section{Sum-product estimates on large sets}

In this section we give a proof of Theorem \ref{Theoremlarge}.
Suppose $A\subset \mathbb{F}_p$ with $|A|\succeq p^{0.5}$.
  Similar to the analysis in Section 3, there
exist   a  subset $Z\subset A$ with $|Z|\geq\frac{|A|}{2}$ such that
\[|Z\pm Z\pm Z\pm Z|\preceq\frac{|A\pm A|^3}{|A|^2},\] and a fixed
element $z_0\in Z$ so that
\[\sum_{z\in Z}|z_0Z\cap
zZ|\geq\frac{|Z|^3}{|ZZ|}\succeq\frac{|A|^3}{|AA|}.\]
 For each
$j\leq\lceil\log_2|Z|\rceil$, let $Z_j$ be the set of all $z\in Z$
for which $|z_0Z\cap zZ|\in N_j$. Choose some
$j_0\leq\lceil\log_2|Z|\rceil$ so that
\[
2^{j_0}|Z_{j_0}|\succapprox\frac{|A|^3}{|AA|}.
\]
There are two cases to consider.

($\spadesuit$) Suppose $\frac{Z_{j_0}-Z_{j_0}}{Z_{j_0}-Z_{j_0}}\neq
\mathbb{F}_p$. By Lemma \ref{Konyaginababcd}, $|Z_{j_0}|\preceq
p^{0.5}$. Similar to (\ref{2222proposition}) one can establish
\[16^{j_0}|Z_{j_0}|^3\preceq|Z\pm Z|^5\cdot|Z\pm Z\pm Z\pm Z|.\]
Consequently,
\[\frac{|A|^{12}}{|AA|^4}\precapprox16^{j_0}|Z_{j_0}|^4\preceq|Z\pm Z|^5\cdot|Z\pm Z\pm Z\pm Z|\cdot|Z_{j_0}|\preceq\frac{|A\pm A|^8}{|A|^2}\cdot p^{0.5},\]
which yields
\begin{equation}\label{recentsss}
|A\pm A|^8|AA|^4\succapprox\frac{|A|^{14}}{p^{0.5}}.
\end{equation}

($\clubsuit$) Suppose $\frac{Z_{j_0}-Z_{j_0}}{Z_{j_0}-Z_{j_0}}=
\mathbb{F}_p$. If $|Z_{j_0}|\preceq p^{0.5}$, then follow the
analysis in ($\spadesuit$) to obtain (\ref{recentsss}). Next suppose
$|Z_{j_0}|\succeq p^{0.5}$. Similar to the proof of
(\ref{2222proposition}) in \cite{Shen,Shen2} one can establish
\[p\preceq(\frac{|Z\pm Z|}{2^{j_0}})^4\cdot|Z\pm Z\pm Z\pm Z|.\]
Consequently,
\[\frac{|A|^{8}}{|AA|^4}\leq\frac{|A|^{12}}{|AA|^4|Z_{j_0}|^4}\precapprox16^{j_0}\preceq\frac{|Z\pm Z|^4|Z\pm Z\pm Z\pm Z|}{p}
\preceq\frac{|A\pm A|^7}{p|A|^2},\] which yields
\begin{equation}\label{cvcvcvcv}
|A\pm A|^7|AA|^4\succapprox|A|^{10}p.
\end{equation} Thus  Theorem
\ref{Theoremlarge} follows from (\ref{recentsss}) and
(\ref{cvcvcvcv}).

\section{Sum-product estimates on different sets}

In this section we  prove Theorem \ref{main11}, Theorem \ref{main33}
and Theorem \ref{main22} together.  Suppose $A,B\subset
\mathbb{F}_p$. Choose a fixed element $a_0\in A$ so that
\[\sum_{a\in A}|aB\cap a_0B|\geq\frac{|A||B|^2}{|AB|}.\]
For each $j\leq\lceil\log_2|B|\rceil$, let $A_j$ be the set of all
$a\in A$ for which $|aB\cap a_0B|\in N_j$. With such preparation and
notations we establish the following proposition (the idea of this
proposition is due to Chun-Yen Shen
\cite{Shendifferentsets,Shen,Shen2}).

\begin{Proposition}\label{3333proposition} (a) If $\frac{A_j-A_j}{A_j-A_j}\neq
\mathbb{F}_p$, then
\begin{equation}\label{nomatter}16^{j}|A_j|^3\preceq\frac{|A+B|^{10}}{|A|^3|B|}.\end{equation}
(b) If $\frac{A_j-A_j}{A_j-A_j}= \mathbb{F}_p$, then
\begin{equation}16^j\cdot\min\{|A_j|^2,p\}\preceq\frac{|A+B|^8}{|A|^3}.\end{equation}
(c) No matter what happens, one always has (\ref{nomatter}) if
$|A_j|\preceq p^{0.5}$.
\end{Proposition}

\begin{proof}

 We only prove this
proposition for the case $\frac{A_j-A_j}{A_j-A_j}\neq \mathbb{F}_p$,
and the interested reader can similarly deal the case
$\frac{A_j-A_j}{A_j-A_j}= \mathbb{F}_p$ and (c) without difficulty.
By Lemma \ref{Konyaginababcd} (if $\frac{A_j-A_j}{A_j-A_j}=
\mathbb{F}_p$, then apply Lemma \ref{lemmallp}), one can find
$a,b,c,d\in A_j$ ($a\neq b$) such that for any ${E}\subset A_j$ with
$|{E}|\geq0.5{|A_j|}$,
\begin{equation}\label{6666}|(b-a){E}+(b-a){E}+(d-c){E}|\succeq |A_j|^2.\end{equation} By Lemma
\ref{coverlemma}, there exist
$$\preceq\frac{|-aA_j-aB\cap a_0B|}{|aB\cap a_0B|}\preceq\frac{|-aA_j-aB|}{2^{j}}\preceq\frac{|A+B|}{2^{j}}$$ translates of $aB\cap
a_0B$ such that these copies can cover $99\%$ of $-aA_j$, there
exist
$$\preceq\frac{|bA_j+bB\cap a_0B|}{|bB\cap a_0B|}\preceq\frac{|bA_j+bB|}{2^{j}}\preceq\frac{|A+B|}{2^{j}}$$ translates of $bB\cap
a_0B$ such that these copies can cover $99\%$ of $bA_j$ . Similar
facts hold for $-cA_j$ and $dA_j$ with corresponding translates of
 $cB\cap
a_0B$ and $dB\cap a_0B$. Hence there exist a subset $A'\subset A_j$
covering $80\%$ of $A_j$, and $\preceq\frac{|A+B|}{2^{j}}$
translates of $aB\cap a_0B$ such that these copies can totally cover
$-aA'$, $\preceq\frac{|A+B|}{2^{j}}$ translates of $bB\cap a_0B$
such that these copies can totally cover $bA'$,
$\preceq\frac{|A+B|}{2^{j}}$ translates of $cB\cap a_0B$ such that
these copies can totally cover $-cA'$, $\preceq\frac{|A+B|}{2^{j}}$
translates of $dB\cap a_0B$ such that these copies can totally cover
$dA'$. Thus
\begin{equation}\label{7777}
|-aA'+bA'-cA'+dA'|\preceq(\frac{|A+B|}{2^{j}})^4|a_0B+a_0B+a_0B+a_0B|.
\end{equation}
By Lemma \ref{ImprovedPlunneckeRuzsa}, there exists an $E\subset A'$
with $|E|\geq0.8|A'|\geq0.64|A_j|\geq0.5|A_j|$ such that
\begin{equation}\label{8888}
|(b-a)E+(b-a)A'+(d-c)A'|\preceq\frac{|A'+A'|}{|A'|}\cdot|(b-a)A'+(d-c)A'|.
\end{equation}
Combining (\ref{6666}), (\ref{7777}) and (\ref{8888}) yields
\begin{equation}
\label{important}
|A_j|^3\preceq|A+A|\cdot(\frac{|A+B|}{2^{j}})^4\cdot|B+B+B+B|.
\end{equation}
Thus we can conclude the proof  by simply applying the
Pl\"{u}nnecke-Ruzsa inequality:
\[
|A+A|\leq\frac{|A+B|^2}{|B|},\ \ \
|B+B+B+B|\leq\frac{|A+B|^4}{|A|^3}.
\]
\end{proof}

\begin{proof}[Proof of Theorem \ref{main11}:] Suppose $|A|\preceq p^{0.5},
|B|\preceq p^{0.5}$. Choose  $j_0\leq\lceil\log_2|B|\rceil$ so that
\[
|A_{j_0}|2^{j_0}\succeq\frac{|A||B|^2}{|AB|\cdot\log_2|B|}.
\]
By Proposition \ref{3333proposition} (c),
\[\frac{|A|^4|B|^8}{|AB|^4\cdot(\log_2|B|)^4}\preceq|A_{j_0}|^4\cdot16^{j_0}\leq|A|\cdot|A_{j_0}|^3\cdot16^{j_0}
\preceq\frac{|A+B|^{10}}{|A|^2|B|},\] which yields
\[|A+B|^{10}|AB|^4\succeq\frac{|A|^6|B|^9}{(\log_2|B|)^4}.\]
By symmetry,
\[|B+A|^{10}|BA|^4\succeq\frac{|B|^6|A|^9}{(\log_2|A|)^4}.\]
This proves Theorem \ref{main11}.
\end{proof}

\begin{proof}[Proof of Theorem \ref{main33}:]   Suppose $|A|\succeq p^{0.5},
|B|\succeq p^{0.5}$. Choose $j_0\leq\lceil\log_2|B|\rceil$ so that
\[
|A_{j_0}|2^{j_0}\succeq\frac{|A||B|^2}{|AB|\cdot\log_2|B|}.
\]
There are two cases to consider.

($\spadesuit$) Suppose $\frac{A_{j_0}-A_{j_0}}{A_{j_0}-A_{j_0}}\neq
\mathbb{F}_p$. By Lemma \ref{Konyaginababcd}, $|A_{j_0}|\preceq
p^{0.5}$. By Proposition \ref{3333proposition} (c),
\[|A_{j_0}|^316^{j_0}\preceq\frac{|A+B|^{10}}{|A|^3|B|}.\]
Consequently,
\[\frac{|A|^{4}|B|^8}{|AB|^4\cdot(\log_2|B|)^4}\preceq16^{j_0}|A_{j_0}|^4=16^{j_0}|A_{j_0}|^3
\cdot|A_{j_0}|\preceq\frac{|A+B|^{10}}{|A|^3|B|}\cdot p^{0.5},\]
which yields
\[|A+B|^{10}|AB|^4\succeq\frac{|A|^7|B|^9}{p^{0.5}\cdot(\log_2|B|)^4}.\]
By symmetry,
\begin{equation}\label{qwer}|B+A|^{10}|BA|^4\succeq\frac{|B|^7|A|^9}{p^{0.5}\cdot(\log_2|A|)^4}.\end{equation}

($\clubsuit$)  Suppose $\frac{A_{j_0}-A_{j_0}}{A_{j_0}-A_{j_0}}=
\mathbb{F}_p$. If $|A_{j_0}|\preceq p^{0.5}$, then follow the
analysis in ($\spadesuit$) to obtain (\ref{qwer}). Next suppose
$|A_{j_0}|\succeq p^{0.5}$, then by Proposition
\ref{3333proposition} (b),
\[p16^{j_0}\preceq\frac{|A+B|^8}{|A|^3}.\]
Hence
\[\frac{|B|^{8}}{|AB|^4\cdot(\log_2|B|)^4}\leq
\frac{|B|^8\cdot|A|^{4}}{|AB|^4\cdot(\log_2|B|)^4\cdot|A_{j_0}|^4}
\preceq16^{j_0}\preceq\frac{|A+B|^8}{p|A|^3},\] which yields
\[|A+B|^8|AB|^4\succeq\frac{p|A|^3|B|^8}{(\log_2|B|)^4}.\]
By symmetry,
\begin{equation}\label{19}|B+A|^8|BA|^4\succeq\frac{p|B|^3|A|^8}{(\log_2|A|)^4}.\end{equation}
Thus Theorem \ref{main33} follows  from (\ref{qwer}) and (\ref{19}).
\end{proof}

\begin{proof}[Proof of Theorem \ref{main22}:]  Suppose $|A|\sim
|B|\preceq p^{0.5}$. By Proposition \ref{3333proposition} (c),
\[
\max_j16^j|A_j|^3\preceq\frac{|A+B|^{10}}{|A|^4}.
\]
By Lemma \ref{lemmachang},
\[
\max_j16^j|A_j|^3\succeq\frac{E^{\times}(A,B)^4}{|A|^5}\succeq\frac{|A|^{11}}{|AB|^4}.
\]
Consequently,
\[
|A+B|^{10}|AB|^4\succeq|A|^{15}.
\]
This proves Theorem \ref{main22}. \end{proof}

\section*{Acknowledgment}
The author would like to thank Chun-Yen Shen for clarifying some
proofs in this paper. He also thanks Jian Shen and Yaokun Wu for
useful discussions. This work was supported by the Mathematical
Tianyuan Foundation of China (No.~10826088) and Texas Higher
Education Coordinating Board (ARP 003615-0039-2007).

\end{document}